\documentclass[12pt,a4paper]{amsart}
\usepackage{color}
\usepackage[utf8]{inputenc}
\usepackage[T1]{fontenc}
\usepackage{amsmath}
\usepackage{bbm}
\usepackage{amsfonts}
\usepackage{amssymb}
\usepackage{amsthm}
\usepackage{amsrefs}
\usepackage{sansmath}
\usepackage{mathrsfs}
\usepackage{extarrows}
\usepackage{tikz-cd}
\usepackage[all]{xy}
\usepackage{mathtools}
 \usepackage[english]{babel}

\theoremstyle{plain}
\newtheorem{thm}{Theorem}
\newtheorem*{thm*}{Theorem}
\theoremstyle{definition}
\newtheorem{defn}[thm]{Definition}
\newtheorem{lem}[thm]{Lemma}
\newtheorem{prop}[thm]{Proposition}
\newtheorem*{cor}{Corollary}

\newtheorem{exmp}[thm]{Example}

\theoremstyle{remark}
\newtheorem*{rem}{Remark}

\newcommand{\trace}{\operatorname{tr}}
\renewcommand{\dim}{\operatorname{dim}}
\newcommand{\id}{\operatorname{Id}}
\newcommand{\EN}{\Delta(\mathbb{N}_0)}

\newcommand{\cF}{\mathcal F}
\newcommand{\R}{{\mathbb R}}

\newcommand{\Z}{{\mathbb Z}}
\newcommand{\N}{{\mathbb N}}

\newcommand{\SC}{\mathcal{SC}_*}
\newcommand{\SCd}{\mathcal{SC}_{**}}
\newcommand{\rrsc}{\emph{rrsc}}
\newcommand{\alg}{\operatorname{Alg}}
\newcommand{\colim}{\underrightarrow{\mathrm{lim}}}
\newcommand{\stab}{\mathrm{Stab}}

\author[M. Schrödl-Baumann]{Michael Schrödl-Baumann}
\address{Institute for Algebra and Geometry \\ Karlsruhe Institute of Technology \\ Englerstr. 2 \\ 76128 Karlsruhe, Germany}
\email{michael.schroedl@kit.edu}
\title[$\ell^2$-Betti numbers of random rooted simplicial complexes]{$\ell^2$-Betti numbers of random rooted simplicial complexes}

\begin{document}

 \maketitle
 
\begin{abstract}
We define unimodular measures on the space of rooted simplicial complexes and associate to each measure a chain complex and a trace function. As a consequence, we can define $\ell^2$-Betti numbers of unimodular random rooted simplicial complexes and show that they are continuous under Benjamini-Schramm convergence.
\end{abstract}
\section*{introduction}
Benjamini and Schramm introduced unimodular random rooted graphs and the notion of local weak convergence in \cite{benjamini-schramm}, which has become a well-studied topology to investigate graphs or networks. Since in interacting complex systems not only interactions between two nodes, but also between more nodes occur, we generalize - building on a work of Elek \cite{elek} - a part of the extensive theory of graphs to simplicial complexes. \\
 A \emph{rooted simplicial complex} is a triple $(K,V(K),x)$ consisting of a simplicial complex $K$ on a vertex set $V(K)$ and a fixed $x\in V(K)$. Usually we omit $V(K)$ and just write $(K,x)$. A simplicial complex is \emph{locally finite} if every vertex is contained in only finitely many simplices. We say that two rooted simplicial complexes $(K,x)$, $(L,y)$ are isomorphic if there is a simplicial isomorphism $\Phi\colon K\to L$ such that $\Phi(x)=y$. Let $\mathcal{SC}_*$ be the space of \emph{isomorphism classes of connected locally finite rooted simplicial complexes} (see Section \ref{rrsc}). Bowen made use of this space in \cite{bowen} but similar ideas can already be found in \cite{elek}. \\ Every finite simplicial complex $K$ defines a \emph{random rooted simplicial complex} $\mu_K$, which is a probability measure on $\SC$, by choosing uniformly at random a vertex $x\in V(K)$ as root (Example \ref{example-measures}). A sequence $K_n$ of simplicial complexes \emph{converges weakly} or \emph{Benjamini-Schramm} if the measures $\mu_{K_n}$ weakly converge (Definition \ref{definition-BSconvergence}). The notion of weak convergence of graphs goes back to Benjamini and Schramm \cite{benjamini-schramm} - for simplicial complexes it was introduced by Elek in \cite{elek}. \\ We will give a definition of $\ell^2$-Betti numbers of unimodular random rooted simplicial complexes (Definition \ref{bettinumbers}) and provide a proof of the following result:
 \begin{thm*} Let $(\mu_n)_{n\in\N}$ be a sequence of sofic random rooted simplicial complexes with uniformly bounded vertex degree. If the sequence weakly converge to a random rooted simplicial complex $\mu$, then the $\ell^2$-Betti numbers of $(\mu_n)_{n\in\N}$ converge to the $\ell^2$-Betti numbers of $\mu$.
\end{thm*} 
As a consequence we get that the limit of normalized Betti numbers of a Benjamini-Schramm convergent sequence $K_n$ of finite simplicial complexes of uniformly bounded degree exists, which is a result of Elek \cite{elek}, and the limit is the $\ell^2$-Betti number of the limit measure.  \\
This result generalizes Lück's approximation Theorem \cite{lueck}, which says in its classical version that for a free $\Gamma$-CW-complex $X$, with $\Gamma\backslash X$ of finite type, 
\[\lim_{i\to\infty}\frac{b_p(\Gamma_i\backslash X;\mathbb{Q})}{[\Gamma:\Gamma_i]}=\beta_p^{(2)}(X,\Gamma)\]
holds for all finite index normal towers $(\Gamma_i)_{i\in\N}$. Farber \cite{farber} extended this to so called Farber sequences, which weakens the assumption that $(\Gamma_i)_{i\in\N}$ is normal. Elek and Szab\'o \cite{elek-szabo} further extended this result to normal towers $(\Gamma_i)_{i\in\N}$ such that $\Gamma/\Gamma_i$ is sofic for every $i\in \N$.
In \cite{7s} Abert, Bergeron, Biringer, Gelander, Nikolov, Raimbault and Samet proved an approximation result for Benjamini-Schramm convergent sequences of orbifolds which arise from a uniformly discrete sequence of lattices in a semi-simple Lie group. Petersen, Sauer and Thom generalized the notion of Farber sequences to lattices in a totally disconnected, second countable unimodular group $G$, which is also a instance of local weak convergence and proved that the normalized Betti numbers of a uniformly discrete Farber sequence converge to the $\ell^2$-Betti number of the group $G$ \cite{petersen-sauer-thom}. Recently Carderi, Gaboriau and de la Salle proved a theorem for graphed groupoids and ultra products of them, which for instance implies the convergence for Farber sequences \cite{carderi}. 
\section{random rooted simplicial complexes}\label{rrsc}
Let $\SC$ be the space of isomorphism classes of connected locally finite rooted simplicial complexes. We denote an element of $\SC$ with representative $(K,x)$ by $[K,x]$. We write $B_{(K,x)}(r)$ for the closed ball of radius $r$ centred at $x$ in $K$, i.e. the rooted subcomplex of $K$ consisting of all simplices of $K$ such that the graph distance of all their vertices to the root is at most $r$. We denote the subset of $\SC$ of all possible rooted isomorphism classes of $r$-balls by $\mathcal{B}(r)$. A metric $d$ on $\SC$ is given by 
\[d([K,x],[L,y])=\inf_{r}\left\lbrace \frac{1}{2^r}\;\middle|\;B_{(K,x)}(r)\cong B_{(L,y)}(r)\right\rbrace\]
and we define the $r$-neighbourhood of $[K,x]$ by 
\begin{align*}
U_r(K,x):&=\lbrace[L,y]\in\SC\;|\; B_{(L,y)}(r)\cong B_{(K,x)}(r)\rbrace \\
&=\Big\lbrace[L,y]\in\SC\;|\; d([L,y],[K,x])\le \frac{1}{2^r}\Big\rbrace.
\end{align*}
With the $r$-neighbourhoods of finite simplicial complexes as basis, $\SC$ becomes a Polish space. The isomorphism classes of finite rooted simplicial complexes are a dense and countable subset of $\SC$. Moreover, the subset $\SC^d$ of isomorphism classes of rooted simplicial complexes with vertex degree bounded by $d\in\N$ is a compact subspace of $\SC$ \cite{bowen}.\\
We call a probability measure on $\SC$ a \emph{random rooted simplicial complex} (\rrsc).\\
Let $\SCd$ be the space of isomorphism classes of connected locally finite doubly rooted simplicial complexes. An element $[K,x,y]$ of $\SCd$ consists of a connected simplicial complex $K$ and an ordered pair of vertices $x,y\in V(K)$. The topology is given by
\begin{align*}
U_r([K,x,y])\coloneqq\left\{[L,p,q]\in\SCd\;\middle|\begin{array}{ll}
(B_{(K,x)}(r),x)\cong (B_{(L,p)}(r),p) \text{ and}\\
(B_{(K,y)}(r),y)\cong (B_{(L,q)}(r),q)\end{array}\right\}.
\end{align*}
 An important class of random rooted simplicial complexes are the \emph{unimodular} ones:
\begin{defn}\label{definition unimodular}
A probability measure $\mu$ on $\SC$ is called \emph{unimodular} if
\[\int_{\SC}\sum_{y\in V(K)}f([K,x,y])d\mu([K,x])=\int_{\SC}\sum_{x\in V(K)}f([K,x,y])d\mu([K,y])\]
for all Borel measurable functions $f\colon\SCd\to \mathbb{R}_{\ge 0}$.
\end{defn}
We define the \emph{maximum degree} of an element $[K,x]\in\SC$ by 
\[\mathrm{deg}_{max}([K,x])=\sup_{x\in V(K)}\{\deg(x)\},\]
where $\deg(x)$ denotes the vertex degree, i.e. the number of 1-simplices of $K$ containing $x$. More generally, we define the $p$-degree $\deg_p(x)$ of $x$ as the number of $p$-simplices of $K$ containing $x$. If the $p$-degree is bounded then also the $(p+1)$-degree is. We will use the notation $K(p)$ for the set of $p$-simplices of $K$.
\begin{defn}
The \emph{degree} $\deg(\mu)$ of an random rooted simplicial complex $\mu$ on $\SC$ is the $\mu$-essential supremum of the maximum degree:
\begin{align*}
&\deg(\mu)=\inf\Big\lbrace d\in\mathbb{R}\;|\; \mu\big\{[K,x]\in\SC\;|\;\mathrm{deg}_{max}([K,x]) >d\big\}=0\Big\rbrace.
\end{align*}
We say a measure $\mu$ is of \emph{bounded degree} if $\deg(\mu)<\infty$.\\
For comparison, the \emph{expected p-degree}  of $\mu$ is 
\begin{align*}
\mathbb{E}(\deg_p)=\int_{\SC}\deg_p(x)d\mu.
\end{align*}
\end{defn}
Now we will give some examples of (unimodular) random rooted simplicial complexes:
\begin{exmp}\label{example-measures}
\begin{enumerate}
\item Every finite connected simplicial complex $K$ defines an unimodular \rrsc\, of bounded degree by:
\begin{align*}
\mu_K:=\sum_{x\in V(K)}\frac{\delta_{[K,x]}}{| V(K)|} ,
\end{align*}
where $\delta_{[K,x]}$ denotes the Dirac measure of the point $[K,x]$ in $\SC$. If $K$ has more than one connected component, we define the associated unimodular \rrsc\, by:
\begin{align*}
\mu_K:=\sum_{x\in V(K)}\frac{\delta_{[K_x,x]}}{| V(K)|}, 
\end{align*}
with $K_x$ the connected component of $x\in V(K)$.
\item There is a more general construction, which also applies to infinite complexes. Let $K$ be a connected locally finite simplicial complex with unimodular automorphism group $\mathrm{Aut}(K)$. Further, let $\{x_1,x_2,...\}$ be a complete orbit section of $V(K)$. Then there exist an unimodular probability measure $\mu_K$ on $\SC$, fully supported on the rooted isomorphism classes of $K$, if
\begin{align*}
c\coloneqq\sum_i \frac{1}{|\mathrm{Stab}(x_i)|}<\infty,
\end{align*}
where $|\cdot|$ denotes the Haar measure.
The unimodular \rrsc\, is given by:
\begin{align*}
\mu_{K}:=\frac{1}{c}\sum_{i}\frac{\delta_{[K,x_i]}}{|\mathrm{Stab}(x_i)|}.
\end{align*}
In \cite[Theorem 3.1]{aldous-lyons} Aldous and Lyons give a proof of this statement for graphs, which also applies to simplicial complexes. They proved even more: The measure $\mu_G$ associated with a graph $G$ is unimodular if and only if the automorphism group $\mathrm{Aut}(G)$ is unimodular. An important special case of this is the following. Let $K$ be as above with a proper, admissible (i.e. the stabilizer of a simplex is the intersection of the stabilizers of its vertices) and cocompact action of a discrete normal subgroup $\Gamma\unlhd \mathrm{Aut}(K)=:G$. Let $\cF$ be a fundamental domain for the $\Gamma$ action and $\{x_1,x_2,...\}$ a complete orbit section for $G$ in $V(K)$. Remark that $G/\Gamma$ is also discrete since $\cF$ is finite. Subsequently, we have the following formula for the \rrsc:
\begin{align*}
\mu_{K}&:=\frac{1}{c}\sum_i \frac{\delta_{[K,x_i]}}{|\mathrm{Stab}_G(x_i)|}=\frac{1}{c}\sum_i \sum_{\substack{g\in G\\ gx\in V(\cF)}}\frac{\delta_{[K,gx_i]}|\mathrm{Stab}_{G/\Gamma}(\Gamma x_i)|}{|\mathrm{Stab}_G(x_i)||G/\Gamma|}\\
&=\frac{1}{c}\sum_{y\in V(\cF)}\frac{\delta_{[K,y]}}{|G/\Gamma| |\mathrm{Stab}_{\Gamma}(y)|}
\end{align*}
And in a similar way we obtain that 
\[c=\sum_{y\in V(\cF)} \frac{|G/\Gamma|}{|\mathrm{Stab}_\Gamma(y)|},\]
hence, with $c(\Gamma)=\sum_{y\in V(\cF)} (|\mathrm{Stab}_\Gamma(y)|)^{-1}$, we deduce that
\begin{align*}
\mu_K:=\frac{1}{c(\Gamma)}\sum _{y\in V(\cF)}\frac{\delta_{[K,y]}}{|\mathrm{Stab}_{\Gamma}(y)|}.
\end{align*}
This shows that $\mu_K$ does not depend on the group or the choice of the fundamental domain. Note that if the $\Gamma$-action is free, $\mu_K$ has the following appearance:
\[\mu_K= \sum _{y\in V(\cF)}\frac{\delta_{[K,y]}}{|V(\cF)|}.\]
\item Probability measure preserving equivalence relations $\mathcal{R}\subset X\times X$ on a standard Borel space $(X,\nu)$ also give rise to random rooted simplicial complexes. Let $\Sigma$ be an $\mathcal{R}$-simplicial complex as introduced by Gaboriau \cite{gaboriau}, i.e. a measurable assignment $X\ni x\mapsto\Sigma_x$ of simplicial complexes together with an $\mathcal{R}$-action: For every $(x,y)\in\mathcal{R}$ there is a simplicial isomorphism $\psi_{x,y}\colon\Sigma_x\to \Sigma_y$ such that $\psi_{x,y}\circ \psi_{y,z}=\psi_{x,z}$ and $\psi_{x,x}=\id_{\Sigma_x}$. Further assume there is a Borel fundamental domain $F$ for the action on the vertices with a countable partition $F=\bigsqcup_{j\in J}F_j$ such that the projection $\pi\colon F\to X$ is injective on each $F_j$ and $\sum_J \nu(\pi(F_j))<\infty$. Then this defines a unimodular \rrsc :
\begin{align*}
\mu_{\Sigma,\mathcal{R}}(A)\coloneqq\frac{\sum_{j\in J}\nu\left(\pi(\{f\in F_j\;|\;[\Sigma_{\pi(f)},f]\in A\})\right)}{\sum_{j\in J}\nu(\pi (F_j))},
\end{align*}
for all Borel sets $A$.
A special case of this is a field of graphs $x\mapsto \Phi_x$ defined by a graphing $\Phi$ of $\mathcal{R}$. In this case the fundamental domain is just the diagonal $(x,x)\in \mathcal{R}$ and the random rooted graph is given by
\begin{align*}
\mu_{\Phi,\mathcal{R}}(A)=\nu\left(\{x\in X\;|\; [\Phi(x),x]\in A\}\right).
\end{align*}
Random rooted graphs defined by graphings were also treated in \cite{aldous-lyons}. For more information about measured equivalence relations see \cite{feldman-mooreI, feldman-mooreII}.
\item The different models of random simplicial complexes, like the Erdös-Rényi model $G(n,p)$, the random $d$-complex $Y_d(n,p)$ of Linial and Meshulam \cite{linial-meshulam}, random flag complexes $X(n,p)$ \cite{kahle} or random geometric complexes (e.g. the random geometric graph of Penrose \cite{penrose}) give rise to unimodular measures on $\SC$ by applying the construction of Example \ref{example-measures}.1. to a random sample.
\end{enumerate}
\end{exmp}

\section{The $\ell^2$-chain complex of $\SC$}
In this section we will define the simplicial $\ell^2$-chain complex of $\SC$. For this, we need to choose a representative for each isomorphism class in a measurable way. This was already done by Aldous and Lyons \cite{aldous-lyons} for rooted networks. Our construction is based on their ideas. We denote the infinite dimensional simplicial complex consisting of all finite subsets of $\mathbb{N}_0$ with $\Delta (\mathbb{N}_0)$. Considered as set, this is just $\mathcal{P}_{<\infty}(\mathbb{N}_0)$. 

Since the space of finite subsets of $\mathbb{N}_0$ is countable, every simplex of $\Delta(\mathbb{N}_0)$ can be identified with a natural number. We do this in the following way (diagonal argument): 
\begin{align*}
\left(\begin{array}{cccccccccc}
\N_0& 0&1&2&3&4&5&6&7&... \\
\EN&  \{0\}&\{1\}&\{0,1\}&\{2\}&\{0,2\}&\{1,2\}&\{0,1,2\}&\{3\}&...
\end{array}\right).
\end{align*}
and denote this map by $\Upsilon\colon \N_0\to \EN$. Subsequently all subsets of $\{0,1,2,...,n\}$ show up in the first $2^{n+1}$ entries. 
Hence, we can identify \emph{the space of locally finite connected subcomplexes of }$\Delta(\mathbb{N}_0)$ with a subset of the space of sequences $\{0,1\}^{\mathbb{N}_0}$,  where $0$ at the n-th position means that the simplex $\Upsilon(n)$ of $\EN$ is not contained in the subcomplex. We give  the space of \emph{locally finite connected simplicial subcomplexes of }$\Delta(\mathbb{N}_0)$ the subspace topology with respect to the product topology on $\{0,1\}^{\mathbb{N}_0}$, i.e. open sets are generated by cylinder sets. We have the following Lemma:

\begin{lem}\label{lem-naturalcomplex} There is a continuous map 
\begin{align*}
\Psi\colon\SC&\longrightarrow \{\text{subspaces of } \Delta(\N_0) \}\subseteq\{0,1\}^{\N_0},\\
[K,x]&\longmapsto (K,x)_\mathbb{N}
\end{align*}
such that $[(K,x)_\N,0]=[K,x]$.
\end{lem}
\begin{proof}
Let $[K,x]\in\SC$. That there is a subcomplex $\Lambda$ of $\EN$ with $(\Lambda, 0)\in[K,x]$ is obvious, we obtain such a subcomplex by enumerating the vertices of $K$ in an arbitrary way and label the root with $0$. Thus, we restrict to this kind of subcomplexes.  Given two such  subcomplexes $\Lambda,\Xi\in\{0,1\}^{\N_0}$. Then $\Lambda<\Xi$ if, considered as sequences, $\Lambda<_{\mathrm{lex}}\Xi$ with respect to the lexicographic order, i.e. we look at the first position where the sequences differ, if $\Lambda$ has a 1 there, we say $\Lambda<\Xi$.  $\Psi$ maps $[K,x]$ to the smallest complex $(K,x)_\mathbb{N}\in\{0,1\}^{\N_0}$ with respect to this order.\vspace*{0.1cm}\\
\textbf{Claim:} There exists a unique minimal subcomplex $(K,x)_\N\subset \EN$ for every $[K,x]\in\SC$.\\
 For a finite simplicial complex it follows from the well-ordering principle that a unique minimal subcomplex of $\EN$ exist. Let $[K,x]$ by the rooted isomorphism class of an infinite simplicial complex $K$ and consider the $r$-ball $B_{(K,x)}(r)$ around $x$ in $K$. We claim that $\Psi(B_{(K,x)}(r))=B_{\Psi(B_{(K,x)}(r+1))}(r)$ and hence the minimal element $(K,x)_\N$ is given by $\colim_{r}\Psi(B_{(K,x)}(r))$.\\
First remark that by the definition of the order, the set of vertices must be the interval $[0,1,...,N_r]\subset\N$, where $N_r:=|V(B_{(K,x)}(r))|$, for both of the subcomplexes. For $\Psi(B_{(K,x)}(r))$ this is obvious. For $B_{\Psi(B_{(K,x)}(r+1))}(r)$ we will show it by induction on the second $r$. So suppose that the vertices of $B_{\Psi(B_{(K,x)}(r+1))}(r-1)$ are the interval $[0,...,N_{r-1}]$ but on the other hand there is an $k\in [N_{r-1}+1,...,N_r]$ which is not contained in $B_{\Psi(B_{(K,x)}(r+1))}(r)$, hence $k$ must be at distance $r+1$ to the root in $\Psi(B_{(K,x)}(r+1))$. Therefore, $B_{\Psi(B_{(K,x)}(r+1))}(r)$ must contain another vertex $l\in[N_r+1,..., N_{r+1}]$. But then all simplices of the form $s\cup k  \in \EN$, with $s\subset\{0,...,N_{r-1}\}$, are not contained in $\Psi(B_{(K,x)}(r+1))$. This is a contradiction to the minimality of $\Psi(B_{(K,x)}(r+1))$, since by interchanging the roles of $k$ and $l$ we would get a smaller complex for $B_{(K,x)}(r+1)$. After this consideration and the definition of the order it is clear that $\Psi(B_{(K,x)}(r))=B_{\Psi(B_{(K,x)}(r+1))}(r)$.\vspace*{0.1cm}\\
\textbf{Claim:} $\Psi$ is continuous.\\
Let $[k]\subset\{0,1\}^{\N_0}$ be the cylinder set with a $1$ at the $k$-th position and let $\Upsilon(k)$ the simplex $\{k_0,...,k_n\}\in\EN)(n)$ with $k_0<k_1...<k_n$. The distance of $k_0$  to $0$ is at most $k_0$ for every minimal subcomplex of $\EN$ which contains $\Upsilon(k)$. Consequently, if the image of a complex contains $\Upsilon(k)$, then it is already in the image of the $(k_0+1)$-ball and hence the preimage consists of all $k_0+1$-neighbourhoods of balls with $\Upsilon(k)$ in their image i.e.
\[\Psi^{-1}([k])=\bigcup_{\substack{[L,y]\in \mathcal{B}(k_0+1) \\
\Upsilon(k)\in \Psi ([L,y])}} U_{k_0+1}(L,y).\]
This is a countable union of open sets, therefore $\Psi$ is continuous.
\end{proof}

\begin{defn}[simplicial $\ell^2$-chains of $\SC$]
Let $\mu$ be a random rooted simplicial complex. We define the $p$-th \emph{simplicial $\ell^2$-chain module of $\SC$ with respect to $\mu$} as the direct integral Hilbert space
\[C_{p}^{(2)}(\SC,\mu)\coloneqq\int_{\SC}^\oplus C_p^{(2)}((K,x)_\mathbb{N})d\mu([K,x]),\]
where $C_p^{(2)}((K,x)_\mathbb{N})$ denotes the space of complex or real valued square-summable simplicial $p$-chains of the simplicial complex $(K,x)_\mathbb{N}$.
\end{defn}
The space $C_p^{(2)}(\SC,\mu)$ consists of $\mu$-equivalence classes of $\mu$-measurable functions $f$ on $\SC$ such that $f([K,x])\in C^{(2)}_p((K,x)_\mathbb{N})$ almost everywhere and $\|f\|\in L^2(\SC,\mu)$.
We call a function $f$ \emph{measurable} if 
\[[K,x]\mapsto \langle f([K,x]), \delta_\sigma ([K,x]) \rangle\]
is measurable for all oriented $p$-simplices $\sigma$ of $\EN(p)$, where
\begin{align*}
\delta_\sigma([K,x])=\left\{\begin{array}{ll}
\sigma & \text{if }\sigma\in (K,x)_\N \\
0 & \text{otherwise.}
\end{array}\right.
\end{align*} 
The set $\{\delta_\sigma\;|\; \sigma \in \EN(p)\}$ is a so called \emph{fundamental sequence for}  $C_p^{(2)}(\SC,\mu)$.
The inner product is given by
\[\langle f,g\rangle_\mu =\int_{\SC} \langle f([K,x]),g([K,x])\rangle_{(K,x)_\N} d\mu([K,x]).\]

\vspace*{0.5cm}
Let $T_{[K,x]}$ be a continuous linear mapping of $C_p^{(2)}((K,x)_\N)$ to $C_q^{(2)}((K,x)_\N)$ for every $[K,x]\in\SC$. We call such a field of linear maps $[K,x]\mapsto T_{[K,x]}$ \emph{measurable} if
\[[K,x]\mapsto \langle T_{[K,x]}\delta_\sigma,\delta_\theta \rangle\]
is measurable for every $\sigma\in \EN(p)$ and $\theta\in\EN(q)$.
\begin{defn}[Decomposable operators]
We call an operator 
\[T\colon C_p^{(2)}(\SC,\mu)\to C_q^{(2)}(\SC,\mu)\] \emph{decomposable} if there is a measurable field $[K,x]\to T_{[K,x]}$ of operators, such that for every element $f\in C_p^{(2)}(\SC,\mu)$ in the domain of $T$ and every $[K,x]\in \SC$, $f([K,x])$ is in the domain of $T_{[K,x]}$, $[K,x]\mapsto \|T_{[K,x]}f([K,x])\|$ is square integrable and 
\[(Tf)([K,x])=T_{[K,x]}f([K,x])\]
\end{defn}
 for $\mu$ almost all $[K,x]\in\SC$.\\
 $T$ is closed (resp. densely defined, self-adjoint) if and only $T_{[K,x]}$ is closed (resp densely defined, self-adjoint) $\mu$ almost everywhere \cite[Theorem 3]{nussbaum}.
\begin{prop}\cite[II.3]{dixmier}
Let $[K,x]\to T_{[K,x]}$ be a measurable field of operators on $\SC$. If the $\mu$-essential supremum $D$ of $[K,x]\mapsto \|T_{[K,x]}\|$ is finite, we call $[K,x]\to T_{[K,x]}$ \emph{$\mu$-essentially bounded}  and the field defines a bounded operator $T$ on $C_p^{(2)}(\SC,\mu)$ with
\[\|T\|=D.\]
\end{prop}
For more details about direct integrals see \cite{dixmier} or \cite{nielsen}.
Suppose we have a Borel probability measure $\mu$ such that the fields of boundary operators 
\[\partial_p\colon [K,x]\mapsto [\partial_{p,[K,x]}\colon C_p^{(2)}((K,x)_\mathbb{N})\to C_{p-1}^{(2)}((K,x)_\mathbb{N})]\]
 and its adjoints
 \[\partial_p^*\colon [K,x]\mapsto [\partial_{p,[K,x]}^*\colon C_{p-1}^{(2)}((K,x)_\mathbb{N})\to C_{p}^{(2)}((K,x)_\mathbb{N})]\]
  are $\mu$-essentially bounded. Therefore, the measurable fields of operators define bounded linear mappings between direct integral Hilbert spaces 
\begin{align*}
\int^\oplus_{\SC}\partial_{p,[K,x]} d\mu=\partial_p\colon C_p^{(2)}(\SC,\mu)\to C_{p-1}^{(2)}(\SC,\mu),\\
\int^\oplus_{\SC}\partial_{p,[K,x]}^* d\mu=\partial_p^*\colon C_{p-1}^{(2)}(\SC,\mu)\to C_{p}^{(2)}(\SC,\mu),
\end{align*}
with operator norm equals the $\mu$-essential supremum.
Since $\partial_{p-1}\circ\partial_p([K,x])=0$ for all $[K,x]\in\SC$ and $p\in\N$ we get a chain complex 
\begin{align*}
...\xrightarrow[]{\partial_{p+1}}C_p^{(2)}(\SC,\mu)\xrightarrow[]{\partial_p}C_{p-1}^{(2)}(\SC,\mu)\xrightarrow[]{\partial_{p-1}}...
\end{align*}
and hence, we can define the \emph{p-th reduced homology of $\SC$ with respect to $\mu$} by 
\[\overline{H}^{(2)}_{p}(\SC,\mu)=\ker \partial_p/\overline{\mathrm{im}\;\partial_{p+1}} .\]
Further, we can define the Laplace operator $\Delta_p=\partial_{p-1}^*\circ\partial_p+\partial_{p+1}\circ\partial_{p+1}^*$ and the direct integral Hilbert space of \emph{harmonic chains} $\mathcal{H}_{p}^{(2)}(\SC,\mu)$. There is, as long as $\Delta_p$ is a bounded operator, an Hodge isomorphism (cf. \cite[p. 192]{eckmann}):
\begin{align*}\overline{H}^{(2)}_{p}(\SC,\mu)\cong\mathcal{H}_{p}^{(2)}(\SC,\mu).
\end{align*}

\begin{prop}
For every random rooted simplicial complex $\mu$ of bounded degree $\deg(\mu)=D$ the boundary operator $\partial_p$, its adjoint $\partial_p^*$  and the Laplace operator $\Delta_p$ are bounded operators on $C_p^{(2)}(\SC,\mu)$ for all $p\in\N$.
\end{prop}
\begin{proof}
First remark that $\|\partial_p\|\le \sqrt{p+1}$ is independent of the vertex degree since every $p$-simplex has $p+1$ faces. Let us look at $\partial_p^*$.  First consider the coboundary operator $\partial^*_p$ on a fixed simplicial complex $K$ with vertex degree bounded by $D$. For every $(p-1)$-simplex $\sigma$ of $K$ its image $\partial_p^*(\sigma)$ consists of at most $D-(p-1)$ \emph{p}-simplices since every vertex, which is contained in $\sigma$, must be contained in the $p$-simplices of the image of $\sigma$ and since the degree is bounded, there can only be $D-(p-1)$ of them. Because $\mu$ is only supported on simplicial complexes with vertex degree $\le D$, the $\mu$-essential supremum $\|\partial_p^*\|\le \sqrt{D-(p-1)}$. Putting the two estimates together, we get
\[\|\Delta_p\|\le 2\sqrt{(p+2)D}.\]
\end{proof}

\section{Trace and spectral measure}
We fix a $p\in\N$ and an unimodular \rrsc\,$\mu$ for this section.
\begin{defn}\label{def-vonneumannalgebra}
The von Neumann algebra $\operatorname{Alg}_p(\mu)$ is given by bounded decomposable operators $T$ on $C_p^{(2)}(\SC,\mu)$ such that for almost all $[K,x]\in\SC$ and all non-rooted and not necessarily order preserving simplicial isomorphisms $(K,x)_\N\to (L,y)_\N$ of subcomplexes of $\EN$  and for all simplicial chains $z,z'\in C_p^{(2)}((K,x)_\N)$ 
\begin{equation}\label{equation-vonneumann algebra on hilbert field}
\langle T_{[K,x]}z,z'\rangle=\langle T_{[L,y]}\phi(z),\phi(z')\rangle.
\end{equation}
holds. 
\end{defn}
It is obvious that this condition is closed under the weak topology. Further, Equation \ref{equation-vonneumann algebra on hilbert field} implies that the operators in $\operatorname{Alg}_p(\mu)$ are independent of the root. Therefore we will sometimes write $T_K$ instead of $T_{[K,x]}$.  \\

We already introduced the fundamental sequence $\{\delta_\sigma\;|\; \sigma \in \EN(p)\}$. This time, we pick only the characteristic functions of oriented $p$-simplices which contain $0$ as a vertex, i.e. 
 \[\mathcal{T}_p=\{\delta_\sigma\;|\; \sigma\in \EN(p)\text{ and }0\in\sigma\},\] 
 and call it the \emph{p-th trace family}. We define a \emph{trace} on the von Neumann algebra $\operatorname{Alg}_p(\mu)$ by
\begin{equation}\label{equation-trace}
\trace_p^\mu(T)\coloneqq\sum_{\tau\in\mathcal{T}_p}\frac{\langle T\tau,\tau\rangle_\mu}{p+1}\in [0,\infty].
\end{equation}
Note that the formula does not depend on the chosen orientation of $\tau$. We verify the trace property:
\begin{align*}
&\sum_{\tau\in \mathcal{T}_p}\langle ST\tau,\tau\rangle_\mu=\sum_{\tau\in \mathcal{T}_p}\langle T\tau,S^*\tau\rangle_\mu\\
=&\int_{\SC} \sum_{\tau\in \mathcal{T}_p} \langle T_K\tau([K,x]),S_K^*\tau([K,x])\rangle d\mu([K,x])\\
=&\int_{\SC}\sum_{\sigma\in K(p)}\sum_{\tau\in \mathcal{T}_p}\langle T_K\tau([K,x]),\sigma\rangle\langle\sigma,S_K^*\tau([K,x])\rangle d\mu([K,x])\\
=&\int_{\SC}\sum_{y\in V(K)}\underbrace{\sum_{\substack{\sigma\in (K,x)_\N(p)\\ y\in\sigma}}\sum_{\tau\in \mathcal{T}_p}\frac{\langle T_K\tau([K,x]),\sigma\rangle\langle S_K\sigma,\tau([K,x])\rangle}{p+1}}_{f([K,x,y])}d\mu([K,x]) 
\end{align*}
Now we use unimodularity (Definition \ref{definition unimodular}):
\begin{align*}
=&\int_{\SC}\sum_{y\in V(K)} f([K,y,x])d\mu([K,x])\\
=&\int_{\SC}\sum_{y\in V(K)}\sum_{\substack{\sigma\in (K,y)_\N(p)\\ x\in\sigma}}\sum_{\tau\in \mathcal{T}_p}\frac{\langle T_K\tau([K,y]),\sigma\rangle\langle S_K\sigma,\tau([K,y])\rangle}{p+1}d\mu([K,x]) \\
=&\int_{\SC}\sum_{y\in V(K)}\sum_{\substack{\sigma\in (K,y)_\N(p)\\ x\in\sigma}}\sum_{\tau\in \mathcal{T}_p}\frac{\langle S_K\sigma,\tau([K,y])\rangle\langle \tau([K,y]),T_K^*\sigma\rangle}{p+1}d\mu([K,x]) \\
=&\int_{\SC}\sum_{\substack{\sigma\in (K,x)_\N(p)\\ 0\in\sigma}}\langle S_K\sigma,T_K^*\sigma\rangle d\mu([K,x]) \\
=&\int_{\SC}\sum_{\tau\in\mathcal{T}}\langle S_K\tau([K,x]),T_K^*\tau([K,x])\rangle d\mu([K,x])=\sum_{\tau\in\mathcal{T}_p}\langle TS\tau,\tau\rangle_\mu
\end{align*}

\begin{prop} The trace $\trace_p^\mu$ is  normal, faithful and semi-finite. If the expected $p$-degree at the root $\mathbb{E}_\mu(\deg_p)$ is finite, then $\trace_p^\mu$  is a finite trace.
\end{prop}
\begin{proof}
Faithfulness follows from Equation (\ref{equation-vonneumann algebra on hilbert field}), Definition \ref{def-vonneumannalgebra}. We check semi-finiteness. Let $A\in\alg(\mu)$ be positive. We define the operators
\[A_{n,[K,x]}\coloneqq\frac{A_{[K,x]}}{\max\{\deg_p(x)-n,1\}}.\]
We will show that $A_n$ weakly converges to $A$ but first, we verify that $\trace_p^\mu(A_n)<\infty$.
\begin{align*}
\trace_p^\mu(A_n)&=\int_{\SC}\sum_{\tau\in\mathcal{T}_p}\frac{\langle A \tau,\tau\rangle}{(p+1)\max\{\deg_p(x)-n,1\}}d\mu([K,x])\\
&\le\int_{\SC}\frac{\|A\|\deg_p(x)}{(p+1)\max\{\deg_p(x)-n,1\}}d\mu([K,x])\\
&\le \frac{\|A\|(n+1)}{p+1}
\end{align*}
We check weak convergence. Let $\sigma,\theta\in C_p^{(2)}(\SC,\mu)$:
\begin{align*}
\langle A_n \sigma,\theta\rangle_\mu=&\int_{\SC}\langle A_n\sigma,\theta\rangle d\mu=\int_{\SC}\frac{\langle A\sigma,\theta\rangle}{\max\{\deg_p(x)-n,1\}}d\mu\\
=&\int_{\substack{[K,x]\in\SC\\ \deg_p(x)\le n}}\langle A\sigma,\theta\rangle d\mu + \int_{\substack{[K,x]\in\SC\\ \deg_p(x)> n}}\frac{\langle A\sigma,\theta\rangle}{\deg_p(x)-n}d\mu \\
\xlongrightarrow{n\to\infty}&\int_{\SC}\langle A\sigma,\theta\rangle d\mu=\langle A\sigma,\theta\rangle_\mu
\end{align*}
In the last step we used the fact that the elements of $\SC$ are locally finite, hence the right summand tends to zero and the left one to $\int_{\SC}\langle A\sigma,\theta\rangle d\mu$.
\end{proof}
\begin{exmp}
We compute the trace of the identity operator $\id_p$ on $C_p^{(2)}(\SC,\mu)$:
\begin{align*}
\trace_p^\mu(\id_p)&=\int_{\SC}\sum_{\tau\in\mathcal{T}_p}\frac{\langle \id_p\tau,\tau\rangle}{p+1}d\mu([K,x])\\
&=\int_{\SC}\frac{\deg_p(x)}{p+1}d\mu([K,x])=\mathbb{E}(\deg_p)/(p+1).
\end{align*}
If $\mu$ is the unimodular \rrsc\, associated to a finite simplicial complex $K$ this is equal to $\frac{|K(p)|}{|V(K)|}$.
\end{exmp}
Let $\mu$ be an unimodular random rooted simplicial complex such that the $p$-th Laplace operator is a bounded operator.
\begin{defn}\label{bettinumbers}
We define the \emph{p-th $\ell^2$-Betti number of $\mu$} as the trace of the projection onto the kernel of the Laplace operator, i.e.
\begin{align*}
\beta_p^{(2)}(\mu)\coloneqq\trace_p^\mu(\mathrm{proj}_{\ker(\Delta_p)}).
\end{align*}
\end{defn}

\begin{exmp}\label{example-betti-numbers}
We compute the $\ell^2$-Betti numbers of the \rrsc s defined in Example \ref{example-measures}. We denote by $b_p(K)$ the $p$-th ordinary Betti number of the simplicial complex $K$ and by $\beta
_p^{(2)}(K,\Gamma)$ the $\ell^2$-Betti numbers of a $\Gamma$-simplicial complex with respect to the group von Neumann algebra. Gaboriau gives a definition of $\ell^2$-Betti numbers of $\mathcal{R}$-simplicial complexes which we denote by $\beta_p(\Sigma,\mathcal{R},\nu)$. If the $\mathcal{R}$-simplicial complex is $p$-connected, the definition is independent of the $\mathcal{R}$-simplicial complex $\Sigma$ and hence it is called the $p$-th $\ell^2$-Betti number $\beta_p(\mathcal{R},\nu)$ of the equivalence realtion $\mathcal{R}$ \cite{gaboriau}.
\begin{enumerate}
\item Given a finite simplicial complex $L$ with its associated unimodular \rrsc\, $\mu_L$ and let $P$ denote the projection onto the kernel of the $p$-th Laplace operator. We compute the $\ell^2$-Betti numbers of $\mu_L$:
\begin{align*}
\beta^{(2)}_p(\mu_L)&=\int_{\SC}\sum_{\tau\in\mathcal{T}_p}\frac{\langle P \tau,\tau\rangle}{p+1}d\mu_L([K,x])\\
&=\sum_{\substack{[K,x]\in\SC\\ K\cong L}}\mu_L([K,x])\sum_{\substack{\sigma \in (K,x)_\N(p) \\ 0\in\sigma}}\frac{\langle P \sigma,\sigma\rangle}{p+1}\\
&=\sum_{\substack{[K,x]\in \SC\\ K\cong L}}\frac{|\{y\in V(L)\;|\;(L,y)\in [K,x]\}|}{|V(L)|}\sum_{\substack{\sigma \in (K,x)_\N(p) \\ 0\in\sigma}}\frac{\langle P \sigma,\sigma\rangle}{p+1}\\
&=\frac{1}{|V(L)|}\sum_{y\in V(L)}\sum_{\substack{\sigma \in L(p) \\ y\in\sigma}}\frac{\langle P \sigma,\sigma\rangle}{p+1}\\
&=\frac{1}{|V(L)|}\sum_{\sigma \in L(p)}\langle P \sigma,\sigma\rangle=\frac{b_p(L)}{|V(L)|}=\beta_p^{(2)}(L,\{e\})
\end{align*}
From line three to four we used the fact that $(L,y)$ and $((K,x)_\N,0)$ are root isomorphic, so 
\[\sum_{\substack{\sigma \in (K,x)_\N(p) \\ 0\in\sigma}}\frac{\langle P \sigma,\sigma\rangle}{p+1}=\sum_{\substack{\sigma \in L(p) \\ y\in\sigma}}\frac{\langle P \sigma,\sigma\rangle}{p+1}.\]
Also note, if we sum over all $p$-simplices containing a fixed vertex and then sum over all vertices, we hit every $p$-simplex exactly p+1 times. We will use this fact constantly.
\item Let $L$ be an infinite connected locally finite simplicial complex with a free, admissible and cocompact action of a discrete group $\Gamma$ with fundamental domain $\cF$.
\begin{align*}
\beta^{(2)}_p(\mu_{L})&=\int_{\SC}\sum_{\tau\in\mathcal{T}_p}\frac{\langle P\tau,\tau\rangle}{p+1}d\mu_{L}([K,x]) \\
&=\sum_{\substack{[K,x]\in\SC\\ K\cong L}}\mu_L([K,x])\sum_{\substack{\sigma \in (K,x)_\N(p) \\ 0\in\sigma}}\frac{\langle P\sigma,\sigma\rangle}{p+1}\\
&=\frac{1}{|\cF|}\sum_{x\in \cF}\sum_{\substack{\sigma \in L(p)\\ x\in\sigma}}\frac{\langle P \sigma,\sigma\rangle}{p+1}\\
&=\frac{1}{|\cF|}\sum_{\sigma \in \cF(p)}\langle P \sigma,\sigma\rangle=\frac{\beta^{(2)}_p(L,\Gamma)}{|\cF|}
\end{align*}
\item Given a probability measure preserving equivalence relation $\mathcal{R}\subset X\times X$ on a standard Borel space $(X,\nu)$. Let $\Sigma$ be a $\mathcal{R}$-simplicial complex with fundamental domain $F=\bigsqcup_{j\in J} F_j$ of $\Sigma^{(0)}$ such that $\pi\colon F_j\to X$ is injective and $\sum_{j\in J} \pi(\nu(F_j))<\infty$. We denote the vertex over $x$ in $F_j$ by $f_j(x)$.
A fundamental domain for the ordered simplices of $\Sigma^{(p)}$ is given by
\begin{align*}
\{ (f,v_1,...,v_p)\in \Sigma^{(p)}\;|\; f\in F,\; v_i\in \Sigma^{(0)}\text{ for }1\le i\le p\}.
\end{align*}
Among these ordered simplices we can choose a representative for each unordered simplex, this gives us a fundamental domain $F'=\bigsqcup_{i\in I} F_i'$ of $\Sigma^{(p)}$ for the action of $\mathcal{R}\times \mathfrak{S}_p$ such that the projection $\pi\colon F'_i\to X$ is injective for each $i\in I$. The characteristic functions of these sets define a total sequence $\sigma_i$ for the direct integral $\int_X^\oplus C^{(2)}_p(\Sigma_x)d\nu$ and hence define a trace \cite[p.120]{gaboriau}. We compute the $\ell^2$-Betti numbers:
\begin{align*}
&\beta^{(2)}_p(\mu_{\Sigma,\mathcal{R}})\\
=&\int_{\SC}\sum_{\tau \in \mathcal{T}_p}\frac{\langle P \tau,\tau\rangle}{p+1}d\mu_{\Sigma,\mathcal{R}}([K,x]) \\
=&\int_X \frac{1}{\sum_{j\in J} \nu(\pi(F_j))}\sum_{j\in J}\sum_{\tau \in \mathcal{T}_p}\chi_{\pi(F_j)}\frac{\langle P \tau,\tau\rangle([\Sigma_x,f_j(x)])}{p+1}d\nu\\
=&\int_X \frac{1}{\sum_{j\in J} \nu(\pi(F_j))}\sum_{j\in J}\sum_{\substack{\sigma \in \Sigma_x^{(p)}\\ f_j(x)\in\sigma}}\chi_{\pi(F_j)}(x)\frac{\langle P \sigma,\sigma\rangle_{\Sigma_x}}{p+1}d\nu\\
=&\frac{1}{\sum_{j\in J} \nu(\pi(F_j))}\int_X \chi_{\pi(F'_i)}(x) \sum_{i\in I}\langle P \sigma_i(x),\sigma_i(x)\rangle_{\Sigma_x}d\nu\\
=&\frac{\dim_\mathcal{R}\mathcal{H}_p^{(2)}(\Sigma)}{\sum_{j\in J} \nu(\pi(F_j))}=\frac{\beta_p(\Sigma,\mathcal{R},\nu)}{\sum_{j\in J} \nu(\pi(F_j))}
\end{align*}
As in the examples before, we used the fact that summing over all $p$-simplices containing one of the vertices of the fundamental domain $F$, is the same as summing $(p+1)$ times over a fundamental domain $F'$ for the unordered $p$-simplices.
\end{enumerate}
\end{exmp}
Since we will need it later, we recall the spectral theorem for self-adjoint operators.
\begin{thm}[Spectral theorem]\label{spectral theorem}
 For every self-adjoint operator $T$ on a Hilbert space $\mathcal{H}$ there exist a unique projection valued measure $E_T\colon\mathrm{Bor}(\R)\to \mathcal{P}(\mathcal{H})$ such that  for all bounded Borel functions $f$ on $\R$
 \[f(T)=\int_\R f(\lambda)dE_T(\lambda).\]

\end{thm}

 The following proposition is from Dykema \cite[Proposition 4.2]{dykema}, it also applies to unbounded operators:
\begin{prop}\label{prop-dykema}
Let $T=\int_{X}^\oplus T(\xi)d\eta(\xi)$ be a  self-adjoint, decomposable operator on the direct integral Hilbert space $\mathcal{H}=\int_X^\oplus \mathcal{H}(\xi)d\eta(\xi)$. For every Borel subset $B\subset \R$ we have
\[E_T(B)=\int_{X}^\oplus E_{T(\xi)}(B)d\eta(\xi).\]
Moreover, for every Borel measurable function $f\colon\R\to\R$, we have
\[f(T)=\int_{X}^\oplus f(T(\xi))d\eta(\xi).\]
\end{prop}   
By applying the spectral theorem, we can define the \emph {spectral measure} of a self-adjoint operator $T$ on $C_p^{(2)}(\SC,\mu)$ as
\begin{align*}
&\nu_T\colon \mathrm{Bor}(\R)\to [0,\infty];\\
&\nu_T(B)=\trace_p^\mu(E_T(B))=\sum_{\tau\in \mathcal{T}_p}\langle E_T(B)\tau,\tau\rangle_\mu/(p+1).
\end{align*}
Note that $\nu_T$ has compact support if $T$ is a bounded operator and further, if $T$ is decomposable, this definition is (by Proposition \ref{prop-dykema}) equivalent to
\begin{align*}
\nu_T(B)&=\sum_{\tau\in \mathcal{T}_p}\int_{\SC}\langle E_{T_{[K,x]}}(B)\tau,\tau\rangle d\mu/(p+1)\\
&=\int_{\SC}\sum_{\substack{\sigma\in K(p)\\ x\in\sigma}}\langle E_{T_{[K,x]}}(B)\sigma,\sigma\rangle/(p+1) d\mu \\
&=:\int_{\SC} \nu_{T_{[K,x]}}(B)d\mu
\end{align*}
Observe that by Theorem \ref{spectral theorem} we also have that:
\begin{align*}
\trace_p^\mu(f(T))&=\sum_{\tau\in\mathcal{T}_p}\langle f(T)\tau,\tau\rangle_\mu/(p+1)\\
&=\sum_{\tau\in\mathcal{T}_p}\int_\R f(\lambda)d\langle E_T(\lambda) \tau,\tau\rangle/(p+1)\\
&=\int_\R f(\lambda)d\nu_T(\lambda).
\end{align*}

\begin{rem} We return to the Laplace operator. The Laplace operator is always closable but it is in general, when the vertex degree is not bounded, not essentially self-adjoint. That is the reason why we will only consider \rrsc s of bounded degree. Before we go on, we will summarize some results about the question, when the Laplace operator is essentially self-adjoint. In \cite{wojciechowski} Wojciechowski proves that the $0$th Laplace operator on a locally-finite graph is essentially self-adjoint. 
In \cite[Proposition 2.2]{bordenave} it is shown that the $0$th Laplace operator for unimodular measures on the space of rooted locally finite graphs $\mathcal{G}_*$ is essentially self-adjoint. In \cite{anne} Anné and Torki-Hamza defined a property called $\chi$-completness for graphs, which implies that the $1$st-Laplace operator is essentially self-adjoint. This property was extended by Chebbi \cite{chebbi} to $2$-dimensional simplicial complexes, where it also implies essentially self-adjointness of the $1$st and $2$nd Laplace operator. Chebbi likewise gives an example of a $2$-dimensional simplicial complex with a non self-adjoint Laplace operator. In \cite{linial-peled} Linial and Peled studied the spectral measures of random simplicial complexes $Y(n,\frac{c}{n})$ and showed that they weakly converge to the spectral measure of a Poisson d-tree, which has a self-adjoint Laplace operator. 
\end{rem}
Even though the most things hold true in the unbounded case, as long as the Laplace operator is self-adjoint, we will assume in the following that $\mu$ is an unimodular random rooted simplicial complex of bounded degree. We denote the projection valued measure of the $p$-th Laplace operator $\Delta_p$ on $C_p^{(2)}(\SC,\mu)$ by $E_{\Delta_p}$.
\begin{defn}
We define the \emph{p-th spectral measure} of an unimodular random rooted simplicial complex of bounded degree $\mu$ as 
\[\nu_p^\mu(S)= \trace_p^\mu(E_{\Delta_p}(S)),\]
 for every Borel set $S\subset \R$.
\end{defn}
As remarked after Proposition \ref{prop-dykema}, this is equivalent to
\[\nu_p^\mu(S)=\int_{\SC} \nu_p^K(S)d\mu,\]
where
\[\nu_p^K(S)=\sum_{\substack{\sigma\in K(p)\\
x\in\sigma}}\frac{\langle E_{\Delta_p}^K(S)\sigma,\sigma\rangle}{p+1}\] is the spectral measure of the $p$-th Laplace operator of $K$ with $E_{\Delta_p}^K$ the corresponding projection valued measure. This approach was for example chosen in \cite{abert-thom-virag} and \cite{linial-peled} in the context of rooted graphs.\\
Note that $\nu_p^\mu$ is not a probability measure, but it is finite as long as the expected $p$-degree at the root $\mathbb{E}_\mu(\deg_p)=\int_{\SC}\deg_p(x)d\mu$ is finite (which is the case if $\mu$ has  bounded degree), since we have 
\begin{align*}
\nu_p^\mu(\mathbb{R})&=\trace_p^\mu(E_p(\mathbb{R}))=\trace_p^\mu(\id_p)=\mathbb{E}_\mu(\deg_p)/(p+1).
\end{align*} 

\begin{rem}
For an unimodular random rooted simplicial complex $\mu$ of bounded degree the $p$-th $\ell^2$-Betti number is
\begin{align*}
\beta_p^{(2)}(\mu)=\nu_p^{\mu}(\{0\}).
\end{align*}
\end{rem}

\section{Benjamini-Schramm convergence and approximation Theorem}

\begin{defn}\label{definition-BSconvergence} We say a sequence $K_n$ of simplicial complexes converges \emph{Benjamini-Schramm} to $K$ if 
\[\mu_{K_n}(U_{[L,y]}(r))\to \mu_{K}(U_{[L,y]}(r))\]
for all $r$ and all finite $[L,y]\in \SC$. This is equivalent to say that the \rrsc s $\mu_{K_n}$ weakly converges to $\mu_K$ i.e.
\[\int_{\SC} f([L,x])d\mu_{K_n}([L,x])\to \int_{\SC} f([L,x])d\mu_{K}([L,x])\]
for all bounded continuous functions $f\colon\SC\to \mathbb{R}$. 
\end{defn}

\begin{defn}
We call a random rooted simplicial complex $\mu$ \emph{sofic} if there is a sequence $(K_n)_n$ of finite connected simplicial complexes converging Benjamini-Schramm to $\mu$.
\end{defn}
It is clear that sofic measures are unimodular since unimodularity is preserved under weak convergence. It is an open question if all unimodular measures are sofic \cite{aldous-lyons}.

\begin{exmp}[Towers of covering spaces]
Let $K$ be a simplicial complex of vertex degree bounded by some $D\in \N$ with a simplicial, proper and cocompact action of a discrete group $\Gamma$. Suppose that $\Gamma$ is residually finite and let $(N_n)_{n\in N}$ be a descending sequence of finite index normal subgroups of $\Gamma$ with $\bigcap_{n\in\N}N_n=\{1\}$. Then the sequence $(K_n)_n:=(K/N_n)_n$ converges Benjamini-Schramm to 
\[\mu_K=\frac{1}{c(\Gamma)}\sum_{x\in\cF^{(0)}} \frac{\delta_{[K,x]}}{|\stab_\Gamma(x)|},\]
where $\cF$ is a fundamental domain for the $\Gamma$-action. This can be seen by the following fact: We fix an $r>0$. Then, for sufficiently large $n\in \N$, the $r$-ball $B_r(K,x)$ is isomorphic to the $r$-ball $B_r(K_n,xN_n)$ for all $x\in V(K)$ and the action of $N_n$ on $K$ is free. Let $\cF(N_n)$ be a fundamental domain for the action of $N_n$ on $K$, with $n$ large enough. Hence, we have for every $\alpha\in \SC$:
\begin{align*}
\mu_{K}(U_\alpha(r))&=\frac{|\{x\in V(\cF(N_n))\mid B_r(K,x)\cong \alpha\}|}{|V(\cF(N_n))|} & \text{ cf. Example \ref{example-measures}\,(2)} \\
&= \frac{|\{x\in V(K_n)\mid B_r(K_n,xN_n)\cong \alpha\}|}{|V(K_n)|}&\\
&=\mu_{K_n}(U_\alpha(r)).
\end{align*} 
Let $K$ be a simplicial complex with a cocompact action of a sofic group, then there exist a Benjamini-Schramm convergent sequence of simplicial complexes converging to $\mu_K$.
\end{exmp}

In the following we will present our version of Lück's approximation Theorem \cite{lueck}. The proof adapts the ideas of \cite{abert-thom-virag} to higher dimensions, where Abért, Thom and Virág proved a similar result for graphs \cite{abert-thom-virag}. They showed that for sofic measures of uniformly bounded degree on the space of rooted graphs, weak convergence implies pointwise convergence of the spectral measures. Our proof relies on their ideas, though we have a different approach, which gives us the possibility to define spectral measures in a more general way.  They also mentioned the possibility of applying their result to simplicial complexes by identifying the $p$-simplices with vertices and putting $\Z$-labelled edges between them, according to the matrix coefficients of the $p$-th Laplace operator.
\begin{thm*}\label{theorem-approximation}
Let $(\mu_n)_{n\in\N}$ be a sequence of sofic random rooted simplicial complexes with uniformly bounded vertex degree. If the sequence weakly converge to a random rooted simplicial complex $\mu$, then the $\ell^2$-Betti numbers of $(\mu_n)_{n\in\N}$ converge to the $\ell^2$-Betti numbers of $\mu$.
\end{thm*}

Before we start with the proof we recall the following lemmas. The first one is a consequence of the approximation theorem of Weierstraß and the second one is known as \emph{Portmanteau theorem}:
\begin{lem}\label{lemma-compact support implies weak convergence}
Let $\mu$ be a Borel measure on $\R$ and let $(\mu_n)_{n\in\N}$ be a sequence of measures. Assume there is a compact set $C$ which contains the support of $\mu_n$ for every $n\in\N$ and assume further that 
\[\lim_{n\to\infty}\int_\mathbb{R}f(t)d\mu_n(t)=\int_\mathbb{R}f(t)d\mu(t)\]
holds for all polynomials $f\in \mathbb{R}[x]$. Then $\mu_n$ weakly converges to $\mu$.
\end{lem}
\begin{lem}\label{lemma- portmanteau}
Let $\nu_n$ be a sequence of finite Borel measures on $\mathbb{R}$ which weakly converge to $\nu$. Then
\begin{enumerate}
\item $\liminf_{n\to\infty} \nu_n(U)\ge \nu(U)$ for all open sets $U\subset\mathbb{R}$,
\item $\limsup_{n\to\infty} \nu_n(A)\le \nu(A)$ for all closed sets $A\subset \mathbb{R}$.
\end{enumerate}
\end{lem}
The following lemma is the first step of the proof of the Theorem:
\begin{lem}\label{lemma-weak convergence} For a weakly convergent sequence $\mu_n\to \mu$ of unimodular \rrsc\, with $\deg(\mu_n)\le D<\infty$ for all $n\in\N$, the corresponding spectral measures $\nu^{\mu_n}_p\to \nu^\mu_p$ weakly converge.
\end{lem}
\begin{proof}
To relax the notation we denote the spectral measures $\nu_n\coloneqq\nu_p^{\mu_n}$ and $\nu\coloneqq\nu_p^{\mu}$. The norm of the Laplace operator is bounded by some constant $R(\mu)\in[0,\infty)$ which depends on the degree of the \rrsc. Since $\deg(\mu_n)$ is uniformly bounded, we know $\|\Delta_p\|_{\mu_n}\le R$ for all $n\in\mathbb{N}$ and hence the spectra of the Laplace operators $\Delta_p$ on $C_p^{(2)}(\SC,\mu_n)$ lie in $[-R,R]$ for all $n\in \mathbb{N}$. Moreover, by Lemma \ref{lemma-compact support implies weak convergence}, it is enough to show that\begin{align*}
\lim_{n\to\infty}\int_\mathbb{R}f(x)d\nu_n(x)=\int_\mathbb{R}f(x)d\nu(x)
\end{align*} 
holds for polynomials $f\in \R[x]$. By linearity we can restrict further to $f(x)=x^r$.
\begin{align*}
\int_{-R}^R \lambda^rd\nu_n(\lambda)&=\sum_{\tau\in \mathcal{T}_p}\int_{\SC}\int_{-R}^R \lambda^rd\langle E_p(\lambda)\tau,\tau\rangle d\mu_n/(p+1)\\
&=\sum_{\tau\in \mathcal{T}_p}\int_{\SC} \langle (\Delta_p)^r\tau,\tau\rangle d\mu_n/(p+1)
\end{align*}
The $r$-th power of $\Delta_p$ only depends on the ($r+1$)-neighbourhood of the root and the weak convergence of $\mu_n$ implies that
\[\lim_{n\to\infty}\mu_n(U_s(\alpha))= \mu(U_s(\alpha))\]
 for all $s$-balls $\alpha \in\mathcal{B}(s)$ and $s\in \mathbb{N}$, so the measures $\mu_n$ restricted to ($r+1$)-neighbourhoods converge and hence
\[\lim_{n\to \infty}\int_{-R}^R x^rd\nu_n=\int_{-R}^R x^rd\nu.\qedhere\]
\end{proof}
\begin{lem}\label{lemma-boundeddegreesequence} For a sofic \rrsc\; $\mu$ of degree bounded by some $D>0$ there is always a sequence of finite simplicial complexes $K_n$ with degree bounded by $D$ such that $\mu_{K_n}$ weakly converge to $\mu$. 
\end{lem}
\begin{proof}
Since $\mu$ is sofic, there exists a sequence $L_n$ of finite simplicial complexes such that $\mu_{L_n}$ weakly converge to $\mu$. This is equivalent to $\mu_{L_n}(U_r(\alpha))\to \mu(U_r(\alpha))$ for all rooted r-balls $\alpha\in\mathcal{B}(r)$ and $r\in \N$. We remove edges of $L_n$ which contain at least one vertex of degree greater then $D$, and all higher simplices containing these edges, until the maximum degree is at most $D$ and denote the new sequence by $K_n$.  We consider the possible kinds of $r$-balls and vertices:
\begin{enumerate}
\item Let $\alpha\in \mathcal{B}(r)$ such that there exists a vertex $x\in V(\alpha)$ with $\deg(x)>D$. In this case we know that $\mu(U_r(\alpha))=0$, since $\deg(\mu)\le D$ and hence $\mu_{L_n}(U_r(\alpha))\to 0$. By construction, we also have that $\mu_{K_n}(U_r(\alpha))= 0$. 
\item Suppose all vertices of $\alpha$ have degree at most $D$. If 
\[(B_{(L_n,x)}(r),x)\cong \alpha,\] there are no simplices in this $r$-ball of $L_n$ to remove and also 
\[(B_{(K_n,x)}(r),x)\cong \alpha.\]
\item The last case that remains is 
\[(B_{(L_n,x)}(r),x)\not\cong \alpha\text{ but }(B_{(K_n,x)}(r),x)\cong \alpha.\] Thus, there must be a vertex $y\in V(B_{(L_n,x)}(r))$ with $\deg(y)>D$. 
\end{enumerate}
Consequently we have for every $r$-ball $\alpha$ and $r\in\N$:
\begin{align*}
&\mu_{L_n}(U_r(\alpha))-\mu_{K_n}(U_r(\alpha))\\
\leq& \mu_{L_n}\left(\left\lbrace U_r(\beta)\;\middle|\begin{array}{c}
\beta\in \mathcal{B}(r) \text{ and } \exists y\in V(\beta)\\ \text{ with } \deg(y)>D \end{array}\right\rbrace\right) \\
= &1-\mu_{L_n}\left(\left\lbrace U_r(\beta)\;\middle|\begin{array}{c}
\beta\in \mathcal{B}(r) \text{ and }  \deg(y)\le D\;\forall y\in V(\beta) \end{array}\right\rbrace\right)\\
\xlongrightarrow{n\to\infty} &1-\mu\left(\left\lbrace U_r(\beta)\;\middle|\begin{array}{c}
\beta\in \mathcal{B}(r) \text{ and }  \deg(y)\le D\;\forall y\in V(\beta) \end{array}\right\rbrace\right)\\=&1-1=0.\\
\Longrightarrow& \mu_{K_n}(U_r(\alpha))\to\mu(U_r(\alpha)).\qedhere
\end{align*}
\end{proof}
We complete the proof of the Theorem:
\begin{proof}
By Lemma \ref{lemma-weak convergence} we know that the spectral measures $\nu_n$ of the sofic bounded degree \rrsc s $\mu_n$ weakly converge to the spectral  measure $\nu$ of $\mu$. First assume $\mu$ is an unimodular \rrsc\, associated with a finite simplicial complex $K$ with vertex degree at most $D$. We pick a basis for $C_p^{(2)}(K)$ by choosing an orientation for every $p$-simplex. Hence we can write the Laplace operator $\Delta_p$ on $C_p^{(2)}(K)$ as a $d\times d$ matrix with integral coefficients, where $d$ is the number of $p$-simplices in $K$. Thus, the characteristic polynomial $q(x)$ of $\Delta_p$ is in $\mathbb{Z}[x]$ and of degree $d$. The product of the non-zero roots of $q(x)$, which are all in $[-\|\Delta_p\|,\|\Delta_p\|]$, 
is one of the coefficients of $q$, so it is in $\mathbb{Z}\backslash \{0\}$ and hence we get
\begin{equation}\label{equation-productgreaterone}
1\le \Pi_{\lambda_i\neq 0}|\lambda_i|.
\end{equation}
The spectral measure of any Borel set $S\subset\mathbb{R}$ has the following form (with $v_i$ the normed eigenvector of $\lambda_i$):
\begin{align*}
\nu(S)&=\sum_{\tau\in\mathcal{T}_p}\int_{\SC} \langle E_{\Delta_p}(S)\tau,\tau\rangle d\mu/(p+1)=\sum_{\sigma\in K(p)}\langle E_{\Delta_p}^K(S)\sigma,\sigma\rangle/|V(K)|\\
&=\sum_{\sigma\in K(p)}\sum_{\lambda_i\in S}\langle \langle\sigma,v_i\rangle v_i,\sigma\rangle/|V(K)|=\sum_{\lambda_i\in S}\frac{\|v_i\|}{|V(K)|}=\frac{|\{{\lambda_i\in S}\}|}{|V(K)|}.
\end{align*}
The second equality is the same computation as in Example \ref{example-betti-numbers}. Thence, we can express the number of eigenvalues in the set $S$ by 
\begin{equation}\label{equiation-numberofeigenvalues}
|\{i\mid\lambda_i \in S\}|=\nu(S)|V(K)|.
\end{equation}
Let $I_0=(-\epsilon,\epsilon)\backslash \{0\}$ for some $1>\epsilon>0$. By Equation \ref{equation-productgreaterone} and \ref{equiation-numberofeigenvalues} we get the following inequality:
\[1\le \Pi_{\lambda_i\neq 0}|\lambda_i|\le \epsilon^{|V(K)|\nu(I_0)}\|\Delta_p\|^d,\]
where we estimated the eigenvalues in $I_0$ by $\epsilon$. Therefore we get
\begin{align*}
0&\le -|V(K)|\nu(I_0)\ln(1/\epsilon) + d \ln(\|\Delta_p\|)\\
\nu(I_0)&\le \frac{d\ln(\|\Delta_p\|)}{|V(K)|\ln(1/\epsilon)}.
\end{align*}
The maximal number of $p$-simplices in a complex with vertex degree bounded by  some $D\in\N$ is $\binom{D}{p}\frac{|V(K)|}{p+1}$, consequently we can make the right side independent of the number of vertices $|V(K)|$:
\[\nu(I_0)\le \frac{\ln(\|\Delta_p\|)\binom{D}{p}}{(p+1)\ln(1/\epsilon)}.\]
Moreover, $\|\Delta_p\|$ is bounded by some number $R(D)$ only depending on the vertex degree $D$. Up to now we looked at an \rrsc\, associated to a finite simplicial complex, but we get the same inequality for any sofic measure $\mu$ of degree bounded by $D$, since by Lemma \ref{lemma-boundeddegreesequence} we know it is the limit of a sequence $\mu_{K_n}$ of finite simplicial complexes, all of degree bounded by $D$. We denote the associated spectral measures by $\nu$ respective $\nu_n$ and apply Lemma \ref{lemma- portmanteau}:
\[\nu(I_0)\leq \liminf_{n\to\infty} \nu_n(I_0)\le \frac{\ln(R(D))\binom{D}{p}}{(p+1)\ln(1/\epsilon)}. \]
After we showed this estimate for every sofic \rrsc\, of bounded degree, we can finish the proof. First, we restate the assumptions of the Theorem. Let $\nu_n$ be a sequence of spectral measures associated to a sequence of sofic \rrsc s $\mu_n$ of degree bounded by $D$ which weakly converge to $\mu$ and let $\nu$ be the spectral measure of $\mu$. We conclude:
\begin{align*}
\limsup_{n\to\infty}\nu_n(\{0\})&\le \nu(\{0\})\le \nu((-\epsilon,\epsilon))\le \liminf_{n\to \infty}\nu_n(-\epsilon,\epsilon)\\
&\le \liminf_{n\to \infty}\nu_n(\{0\})+ \frac{\ln(R(D))\binom{D}{p}}{(p+1)\ln(1/\epsilon)}
\end{align*}
Letting $\epsilon$ tend to zero, completes the proof.
\end{proof}
\begin{cor}[Euler-Poincaré Formula]
Let $\mu$ be a sofic random rooted simplicial complex of bounded degree and with dimension $n$. Then
\begin{align*}
\sum_{p=0}^n(-1)^p \beta_p^{(2)}(\mu)=\sum_{p=0}^n(-1)^p\frac{\mathbb{E}(\deg_p)}{p+1},
\end{align*}
where $\mathbb{E}(\deg_p)$ denotes the expected number of $p$-simplices containing the root.
\end{cor}
\begin{proof}
For a random rooted simplicial complex $\mu_K$, associated with a finite simplicial complex $K$, we have by Example \ref{example-betti-numbers} (1) that $\beta^{(2)}_p(\mu_K)=\frac{b_p(K)}{|K^{(0)}|}$ and therefore
\begin{align*}
\sum_{p=0}^n(-1)^p \beta_p^{(2)}(\mu)&=\sum_{p=0}^n(-1)^p\frac{b_p(K)}{|K^{(0)}|}=\frac{\chi(K)}{|K^{(0)}|}=\sum_{p=0}^n(-1)^p\frac{|K(p)|}{|K^{(0)}|}\\
&=\sum_{p=0}^n(-1)^p\int_{\SC} \frac{\deg_p(x)}{p+1}d\mu_K=\sum_{p=0}^n(-1)^p \frac{\mathbb{E}_{\mu_K}(\deg_p)}{p+1},
\end{align*}
where we denote by $\chi(K)$ the Euler characteristic of $K$. Applying our main theorem yields the equality for sofic random rooted simplcial complexes of bounded degree.
\end{proof}

\end{document}